\documentclass{amsart}

\usepackage[left=1.72in,top=1.46in,right=1.72in,bottom=1.46in]{geometry}

\usepackage{amsmath}
\usepackage{amssymb}
\usepackage{amsfonts}
\usepackage{amsopn}
\usepackage{amsthm}

\usepackage{graphicx}

\usepackage{mathrsfs}

\swapnumbers
\theoremstyle{plain}
\newtheorem{thm}{Theorem}[section]
\newtheorem{lem}[thm]{Lemma}
\newtheorem{cor}[thm]{Corollary}
\newtheorem{prop}[thm]{Proposition}
\newtheorem*{conj}{Conjecture}
\newtheorem*{ques}{Question}

 \newtheorem{thmintro}{Theorem}

\theoremstyle{definition}
\newtheorem{ntt}[thm]{}
\newtheorem{ex}[thm]{Example}
\newtheorem{eg}[thm]{Example}
\newtheorem{rem}[thm]{Remark}
\newtheorem{dfn}[thm]{Definition}
\newtheorem{defn}[thm]{Definition}
\newtheorem{fact}[thm]{Fact}

\numberwithin{equation}{section}

\newcommand{\B}{\mathcal{B}}       
\newcommand\PO{\mathcal{P}}    

\newcommand{\zz}{\mathbb{Z}}       
\newcommand{\Z}{\zz}

\newcommand{\R}{\mathbb{R}}
\newcommand{\Q}{\mathbb{Q}}
\newcommand{\F}{\mathbb{F}}
\newcommand{\nn}{\mathbb{N}}       

\newcommand{\cc}{\mathfrak{c}}     

\newcommand{\Gb}{\bar{G}}

\newcommand{\sheaf}[1]{\mathscr{#1}}
\newcommand\LL{\sheaf{L}}       

\renewcommand{\P}{\mathbb{P}} 

\newcommand{\category}[1]{\mathsf{#1}}
\newcommand{\Db}{\category{D}^{\mathrm{b}}} 
\newcommand{\Dperf}{\category{D}^{\mathrm{perf}}} 

\newcommand{\Spec}{\mathrm{Spec}\,}
\newcommand{\car}{\mathrm{char}\,}
\newcommand{\tensor}{\otimes}
\newcommand{\isom}{\cong}
\newcommand{\dual}{^{\vee}}

\newcommand{\ol}[1]{\overline{#1}}
\newcommand{\linedef}[1]{\textit{#1}} 

\DeclareMathOperator{\SL}{SL}
\DeclareMathOperator{\SO}{SO}
\DeclareMathOperator{\Sp}{Sp}

\DeclareMathOperator{\Pic}{Pic}
\DeclareMathOperator{\Ext}{Ext}
\DeclareMathOperator{\Hom}{Hom}

\newcommand{\la}{\lambda}
\renewcommand{\deg}{$^\circ$}
\newcommand{\case}[1]{\smallskip\underline{\emph{Case #1:}}}

\usepackage{hyperref}
\hypersetup{pdftitle={Exceptional collections of line bundles on projective homogeneous varieties}}
\hypersetup{pdfauthor={Alexey Ananyevskiy, Asher Auel, Skip Garibaldi and Kirill Zainoulline}}
\hypersetup{colorlinks=true,linkcolor=blue,anchorcolor=blue,citecolor=blue,linktoc=page}

\begin{document}

\title[Exceptional collections of line bundles]{Exceptional
  collections of line bundles on projective homogeneous varieties}

\author[Ananyevskiy, Auel, Garibaldi, and Zainoulline]{Alexey Ananyevskiy, Asher Auel, Skip Garibaldi and Kirill Zainoulline}

\address{(Alexey Ananyevskiy)   Department of Mathematics and Mechanics, St.~Petersburg State University, Universitetskiy Prospekt 28, St.~Petersburg, 198504, Russia}
\address{(Asher Auel and Skip Garibaldi) Department of Mathematics \& Computer Science, Emory University, Atlanta, GA 30307, USA}
\address{(Kirill Zainoulline) Department of Mathematics and Statistics, University of Ottawa, 585~King Edward, Ottawa ON K1N6N5, Canada}

\subjclass[2010]{Primary 14F05; Secondary 
14M15, 17B22, 20G15}
\keywords{}

\thanks{\texttt{Version of \today.}}


\begin{abstract}
We construct new examples of exceptional collections of line bundles
on the variety of Borel subgroups of a split semisimple linear
algebraic group $G$ of rank 2 over a field.  We exhibit exceptional
collections of the expected length for types $A_2$ and $B_2 = C_2$ and
prove that no such collection exists for type $G_2$.  This settles the
question of the existence of full exceptional collections of line
bundles on projective homogeneous $G$-varieties for split linear
algebraic groups $G$ of rank at most 2.
\end{abstract}

\maketitle

\setcounter{tocdepth}{1}
\tableofcontents


\section*{Introduction}

The existence question for full exceptional collections in the bounded
derived category of coherent sheaves $\Db(X)$ of a smooth projective
variety $X$ goes back to the foundational results of Be{\u\i}linson
\cite{beilinson1}, \cite{beilinson2} and
Bern\v{s}te{\u\i}n--Gelfand--Gelfand \cite{BGG} for $X=\P^n$.  The
works of Kapranov \cite{kapranov:Grassman1},
\cite{kapranov:Grassmann2}, \cite{kapranov:quadric},
\cite{kapranov:homogeneous} suggested that the structure of projective
homogeneous variety on $X$ should imply the existence of full
exceptional collections.

\begin{conj}
 Let $X$ be a projective homogeneous variety of a split semisimple
 linear algebraic group $G$ over a field of characteristic zero. Then
 there exists a full exceptional collection of vector bundles in
 $\Db(X)$.
\end{conj}

This conjecture remains largely unsolved; see \cite[\S1.1]{kuzpol} for
a recent survey of known results.  The bounded derived category
$\Db(X)$ has come to be understood as a homological replacement for
the variety $X$; exceptional collections provide a way to break up
$\Db(X)$ into simple components.  Such decompositions of the derived
category can be seen as analogous to decompositions of the motive of
$X$, a relationship that has been put into a conjectural framework by
Orlov~\cite{orlov:motives}.  As an example, the existence of a full
exceptional collection implies a splitting of the Chow motive
$M(X)_{\Q}$ into twists of Lefschetz motives \cite{marcolli_tabuada}.  We remark that such a motivic decomposition
is already known for projective homogeneous varieties of split linear
algebraic groups \cite{kock}.

Let $X$ be a variety over a field $k$.  An object $E$ of $\Db(X)$ is
called \linedef{exceptional} if $\Ext^*(E,E)=k$, cf.\
\cite[Def.~1.1]{GR87}.  Let $W$ be a finite set and let $\PO$ be a
partial order on $W$.  An ordered set (with respect to $\PO$) of
exceptional objects $\{E_w\}_{w\in W}$ in $\Db(X)$ is called a {\em
$\PO$-exceptional collection} if
$$
\Ext^*(E_w,E_{w'})=0\text{ for all }w <_\PO w'.
$$
If $\PO$ is a total order, then a $\PO$-exceptional collection is
simply called an \linedef{exceptional collection}.  A
$\PO$-exceptional collection $\{E_w\}_{w\in W}$ is called
\linedef{full} if the smallest triangulated category containing
$\{E_w\}_{w\in W}$ is $\Db(X)$ itself.  Finally, a $\PO$-exceptional
collection of vector bundles $\{E_w\}_{w\in W}$ is said to be of the
\linedef{expected length} if the cardinality of $W$ equals the minimal
number of generators of the abelian group $K_0(X)$.  Note that any
full $\PO$-exceptional collection of vector bundles is of the expected
length.  It is expected that the converse also holds
\cite[Conj.~9.1]{kuznetsov:hochschild}.

In the present paper we address the following closely related question:

\begin{ques}
Let $X$ be the variety of Borel subgroups of a split semisimple
linear algebraic group $G$ and fix a partial order $\PO$ on the Weyl group $W$ of $G$.
Does $\Db(X)$ have a $\PO$-exceptional collection
of the expected length consisting of line bundles?
\end{ques}

On the one hand, the question strengthens the
conjecture by requiring the collection to consist of line bundles.
On the other hand, it 
weakens the conjecture by
allowing partial orders (such as the weak or strong Bruhat orders)
instead of a total order and allowing the collection to merely
generate $K_0(X)$.

So far, a natural way to propagate known exceptional collections of
line bundles is to use the result of Orlov \cite[Cor.~2.7]{Orlov},
that $\Db(X)$ has a full exceptional collection of line bundles if
there exists a (Zariski locally trivial) projective bundle $X\to Y$
such that $\Db(Y)$ has a full exceptional collection of line bundles.
More generally, $\Db(X)$ has a full exceptional collection of line
bundles if $X$ is the total space of a smooth Zariski locally trivial
fibration, whose fiber and base both have derived categories with full
exceptional collections of line bundles \cite{CdRM-R}.  We remark that
the result of Orlov on semiorthogonal decompositions of projective
bundles (hence that of Be{\u\i}linson and
Bern\v{s}te{\u\i}n--Gelfand--Gelfand on $\P^n$) holds over an
arbitrary noetherian base scheme, cf.\ \cite[\S11]{Walter}.  Using
such techniques, one immediately answers the question positively for
all projective homogeneous varieties $X$ associated to split
semisimple groups of rank 2, except in the following cases:\ type
$B_2=C_2$ and $X$ a 3-dimensional quadric; and type $G_2$ and all $X$
(this includes the case of 5-dimensional quadrics).  These exceptions
can be viewed as key motivating examples for the present paper.

It is not necessarily expected that a $\PO$-exceptional collection of
line bundles of the expected length exists for every projective
homogeneous variety.  For instance, certain Grassmannians of type
$A_n$ have full exceptional collections of vector bundles, but not of
line bundles.  This distinction is highlighted in
\cite[Problem~1.2]{CdRM-R}.  Observe also that if $K_0(X)$ is not
generated by line bundles, then $\Db(X)$ can not possess a
$\PO$-exceptional collection of line bundles of the expected length
for obvious reasons.  For example, we prove such a result in
Proposition~\ref{propexist}.  Finally, $\PO$-exceptional collections
may exist, while exceptional collections may not.

\medskip
We will now outline our main results. The paper consists of two parts.
In Part~\ref{positive}, using $K_0$ techniques, we introduce a new
purely combinatorial algorithm for constructing $\PO$-exceptional
collections of line bundles based on a different description of the
Steinberg basis \cite{steinberg} obtained in \cite{An}.  We apply it
to show the following:

\begin{thmintro}
\label{posAB} 
Let $X$ be the variety of Borel subgroups of a split semisimple linear
algebraic group $G$ of rank 2 over a field of characteristic
zero. Then $\Db(X)$ has a $\PO$-exceptional collection of the expected
length consisting of line bundles, for $\PO$ a partial order
isomorphic to the left weak Bruhat order on the Weyl group of $G$.
\end{thmintro}

For example, $\Db(X)$ possesses such a $\PO$-exceptional collection
for the variety $X$ of complete flags of a split group of type $G_2$.

\medskip

In \S\ref{summary.sec}, using combinatorial and
geometric arguments we settle the question of the existence of full
exceptional collections of line bundles on projective homogeneous
$G$-varieties for every split semisimple $G$ of rank $\le 2$ over an
arbitrary field. The crux case here is that of type $G_2$, to which
the entirety of Part~\ref{negative} is devoted:

\begin{thmintro} 
\label{neg} 
None of the three non-trivial projective homogeneous varieties of a
simple algebraic group of type $G_2$ have an exceptional collection
of the expected length consisting of line bundles.
\end{thmintro}

\noindent {\small \textbf{Acknowledgments.} 
We are grateful to Alexander Kuznetsov for comments and for explaining us
some facts concerning exceptional collections.  This research was
partially supported by NSF grant DMS-0903039 (Auel); NSA grant
H98230-11-1-0178 and the Charles T.~Winship Fund
(Garibaldi); and  NSERC Discovery 385795-2010,
Accelerator Supplement 396100-2010, and the Early Researcher Award
grants (Zainoulline).}


\part{Preliminaries and existence results} \label{positive}

\section{Weights and line bundles}

In the present section we recall several basic facts concerning root
systems, weights, associated line bundles, and the Grothendieck group
$K_0$; see \cite{Bou:g7}, \cite{demazure:Schubert}, \cite{FH},
\cite{Panin:twisted}.

\begin{ntt} 
Let $G$ be a split simple simply connected linear algebraic group of
rank $n$ over a field $k$.  We fix a split maximal torus $T$ and a
Borel subgroup $B$ such that $T\subset B\subset G$.

Let $\Lambda$ be the weight lattice
of the root system $\Phi$ of $G$. 
Observe that $\Lambda$ is the group of characters of $T$.
Let $\Pi=\{\alpha_1,\ldots,\alpha_n\}$ be a set of simple roots
and let 
$\{\omega_1,\ldots,\omega_n\}$ be the respective 
set of fundamental weights (a basis of $\Lambda$), 
i.e., $\alpha_i\dual(\omega_j)=\delta_{ij}$.
Let $\Phi^+$ denote the set of all positive roots and let $\Lambda^+$ denote
the cone of dominant weights.
\end{ntt}

\begin{ntt}
Consider the integral group ring $\zz[\Lambda]$; 
its elements are finite linear combinations 
$\sum_i a_ie^{\lambda_i}$, $\lambda_i\in \Lambda$, $a_i\in\zz$. 
Observe that $\zz[\Lambda]$ can be identified with the representation ring of $T$.
Let $X = G/B$ denote the variety of Borel subgroups of $G$, i.e., 
the variety of subgroups conjugate to $B$.

Consider the characteristic map for $K_0$, 
\[
\cc\colon \zz[\Lambda] \to K_0(X)
\]
defined by sending $e^\lambda$ to the class of the associated
homogeneous line bundle $\LL(\lambda)$ over $X$; see
\cite[\S2.8]{demazure:Schubert}.  It is a surjective ring homomorphism
with kernel generated by augmented invariants.  More precisely, if
$\zz[\Lambda]^W$ denotes the subring of $W$-invariant elements and
$\epsilon\colon\zz[\Lambda]\to \zz,\; e^{\lambda}\mapsto 1$ is the
augmentation map, then $\ker \cc$ is generated by elements $x\in
\zz[\Lambda]^W$ such that $\epsilon(x)=0$.  In particular, the Picard
group $\Pic(X)$ coincides with the set of homogeneous line bundles $\{
\LL(\lambda)\}_{\lambda \in \Lambda}$.
\end{ntt}

\begin{ntt} \label{basweights}
The Weyl group $W$ acts linearly on $\Lambda$
via simple reflections $s_i$ as
\[
s_i(\lambda)=\lambda - \alpha_i\dual(\lambda)\alpha_i,\quad \lambda\in\Lambda.
\]
Let $\rho$ denote the half-sum of all positive roots; it is also the
sum of the fundamental weights \cite[VI.1.10,~Prop.~29]{Bou:g4}.

Following \cite{An}, for each $w\in W$ consider the cones $\Lambda^+$
and $w^{-1}\Lambda^+$.  Let $H_\alpha$ denote the hyperplane
orthogonal to a positive root $\alpha\in \Phi^+$.  We say that
$H_\alpha$ separates $\Lambda^+$ and $w^{-1}\Lambda^+$ if
\[
\Lambda^{+}\subset \{\lambda\in \Lambda \mid \alpha\dual(\lambda)\ge 0\}\text{ and } w^{-1}\Lambda^{+}\subset \{\lambda\in \Lambda\,|\,\alpha\dual(\lambda)\le 0\}.
\]
or, equivalently, if $\alpha\dual(w^{-1}\rho)<0$.
Let $H_w$ denote the union of all such hyperplanes, i.e.,
\[
H_w=\bigcup_{\alpha\dual(w^{-1}\rho)<0} H_\alpha.
\]
Consider  the set $A_w= w^{-1}\Lambda^+ \setminus H_w$
consisting of weights $\lambda\in w^{-1}\Lambda^+$ 
separated from $\Lambda^+$ by the same set of hyperplanes as $w^{-1}\rho$. 
By \cite[Lem.~6]{An} there is a unique element $\lambda_w\in A_w$ 
such that for each $\mu \in A_w$ we have $\mu - \lambda_w \in w^{-1}\Lambda^+$.
In fact, the set $A_w$ can be viewed as a cone $w^{-1}\Lambda^+$ shifted 
to the vertex $\lambda_w$.
\end{ntt}

\begin{ex}
In particular, for the identity $1\in W$ we have $\lambda_1=0$. 
Let $w=s_j$ be a simple reflection, then 
\[
w^{-1}\Lambda^+=
\nn_0s_j(\omega_j)\oplus \bigoplus_{i\neq j}\nn_0\omega_i =
\nn_0(\omega_j-\alpha_j)\oplus \bigoplus_{i\neq j}\nn_0\omega_i 
\] 
and
$A_w=(\omega_j-\alpha_j) + w^{-1}\Lambda^+$.
Hence, in this case we have $\lambda_w=\omega_j-\alpha_j$. 

For the longest element $w_0\in W$ we have $\lambda_{w_0}=-\rho$.
\end{ex}

\begin{ntt}
By \cite[Thm.~2]{An}, the integral group ring $\zz[\Lambda]$ is a free
$\zz[\Lambda]^W$-module with the basis $\{e^{\lambda_w}\}_{w\in W}$.
As there is an isomorphism
\[
\zz[\Lambda]\otimes_{\zz[\Lambda]^W}\zz=\zz[\Lambda]/\ker\cc\simeq 
K_0(X),
\]
classes of the associated homogeneous line bundles
$\cc(e^{\lambda_w})=[\LL(\lambda_w)]$ for $w\in W$ generate $K_0(X)$.
\end{ntt}


\section{$\PO$-exceptional collections and scalar extension}

In this section, we assemble some results concerning the interaction
between flat base change and semiorthogonal decompositions in order to
reduce questions concerning exceptional collections on smooth
projective varieties over $k$ to over the algebraic closure $\ol{k}$.
If $X$ is a variety over $k$, write $\ol{X}$ for
the base change  $X \times_k \ol{k}$, and similarly for complexes of sheaves on
$X$.

Let $W$ be an finite set and $\PO$ be a partial order on $W$.

\begin{prop}
\label{prop:scalar_extensions}
Let $X$ be a smooth projective variety over a field $k$ and let
$\{E_w\}_{w \in W}$ be a $\PO$-ordered set of line bundles on
$X$. Then $\{E_w\}$ is a $\PO$-exceptional collection of $\Db(X)$ if
and only if $\{\ol{E}_w\}$ is a $\PO$-exceptional collection of
$\Db(\ol{X})$.  Moreover, $\{E_w\}$ is full if and only if
$\{\ol{E}_w\}$ is full.
\end{prop}
\begin{proof}
For any coherent sheaves $E$ and $F$ on $X$,
we have $\Ext^*(E,F) \tensor_k \ol{k} \isom \Ext^*(\ol{E},\ol{F})$ by
flat base change.  In particular, $E$ is an exceptional object of
$\Db(X)$ if and only if $\ol{E}$ is an exceptional object of
$\Db(\ol{X})$.  Also, for each $w <_{\PO} w'$, we have that
$\Ext^*(\ol{E}_w,\ol{E}_{w'}) = 0$ if and only if
$\Ext^*(E_w,E_{w'})=0$.  Thus $\{E_w\}$ is a $\PO$-exceptional
collection of $\Db(X)$ if and only if $\{\ol{E}_w\}$ is a
$\PO$-exceptional collection of $\Db(\ol{X})$.

Suppose that $\{E_w\}$ is a full $\PO$-exceptional collection of
$\Db(X)$.  As $X$ is smooth, $\Db(X)$ is equivalent to the derived
category $\Dperf(X)$ of perfect complexes on $X$, and similarly for
$\ol{X}$.  As $\{E_w\}$ are line bundles, they are perfect complexes
and $\{ E_w \}$ is a $\PO$-exceptional collection of $\Dperf(X)$.
The main results of \cite{kuznetsov:base_change} imply that
$\{\ol{E}_w\}$ is a full $\PO$-exceptional collection of
$\Dperf(\ol{X})=\Db(\ol{X})$.  Indeed, $\Spec \ol{k} \to
\Spec k$ is faithful and the proof 
in \cite[Prop.~5.1]{kuznetsov:base_change} 
immediately generalizes to the $\PO$-exceptional setting, showing that
$\{\ol{E}_w\}$ is a full $\PO$-exceptional collection.

Now, suppose that $\{\ol{E}_w\}$ is a full $\PO$-exceptional
collection of $\Db(\ol{X})$.  Let $\category{E}$ be the triangulated
subcategory of $\Db(X)$ generated by the $\PO$-exceptional collection
$\{E_w\}$ and $\category{J}$ be the orthogonal complement of
$\category{E}$, i.e., there is a semiorthogonal decomposition $\Db(X)
= \langle \category{E},\category{J} \rangle$.  By
\cite[Prop.~2.5]{kuznetsov:resolutions}, the projection functor
$\Db(X)=\Dperf(X) \to \category{J}$ has finite cohomological
amplitude, hence by \cite[Thm.~7.1]{kuznetsov:base_change}, is
isomorphic to a Fourier--Mukai transform for some object $K$ in
$\Db(X\times_k X)$.  However, $\ol{\category{J}}=0$ by assumption,
hence $\ol{K}=0$.  This implies that $K=0$, otherwise, $K$ would have
a nonzero homology group, which would remain nonzero over $\ol{k}$ by
flat base change.  Thus $\category{J}=0$ and so $\{E_w\}$ generates
$\Db(X)$.
\end{proof}

In this paper, we are concerned with the case where $X$ is a
projective homogeneous variety under a linear algebraic group $G$ that
is split semisimple or has type $G_2$.  Under this hypothesis, the
pull back homomorphism $K_0(X) \to K_0(\ol{X})$ is an isomorphism
\cite{Panin:twisted}.  Together with
Proposition~\ref{prop:scalar_extensions}, this reduces the question of
(non)existence of an exceptional collection of expected length on $X$
to the same question on $\ol{X}$.


\section{$\PO$-exceptional collections on Borel varieties}

Consider the variety $X$ of Borel subgroups of a split semisimple simply
connected linear algebraic group $G$ over a field $k$.  Let
$\LL(\lambda)$ be the homogeneous line bundle over $X$ associated to
the weight $\lambda$.  Recall that $\la$ is \emph{singular} if $\alpha^\vee(\la) = 0$ for some root $\alpha$.

\begin{prop} \label{BWB}
Let $W$ be a finite set endowed with a partial order $\PO$ and let $\{ \la_w \}_{w \in W}$ be a set of weights indexed by $W$.  The statements:
\begin{enumerate}
\item \label{BWB.bundles} $\{ \LL(\la_w) \}_{w \in W}$ is a $\PO$-exceptional collection of line bundles on the Borel variety $X$.
\item \label{BWB.wts} $\la_{w'} - \la_w + \rho$ is a singular weight for every $w {<_{\PO}} w'$.
\end{enumerate}
are equivalent if $\car k = 0$.  If $\car k > 0$, then \eqref{BWB.bundles} implies \eqref{BWB.wts}.
\end{prop}

\begin{proof}
Since $X$ is smooth, proper, and irreducible, we have
$\Hom(\LL,\LL)=k$ for any line bundle $\LL$ over $X$.
Furthermore, for any weights $\la, \la'$, we have
\begin{equation} \label{BWB.1}
\Ext^i(\LL(\lambda),\LL(\lambda'))=H^i(X,\LL(\lambda)\dual\otimes \LL(\lambda'))=H^i(X,\LL(\lambda'-\lambda)).
\end{equation}
In particular, since $0$ is a dominant weight,
Kempf's vanishing theorem \cite[Prop.~II.4.5]{jantzen} implies that $\LL(\la)$ is an exceptional
object for every weight $\la$. 

Equation \eqref{BWB.1} says that \eqref{BWB.bundles} is equivalent to:
$H^*(X, \LL(\la_{w'} - \la_w)) = 0$ for all $w {<_{\PO}} w'$.  If
$\car k = 0$, this is equivalent to \eqref{BWB.wts} by the
Borel-Weil-Bott Theorem as in \cite[p.392]{FH} or
\cite[II.5.5]{jantzen}.  If $\car k = p > 0$, then $X$ and every line
bundle $\LL(\mu)$ is defined over the field $\F_p$ and can be lifted
to a line bundle over $\Q$, via a smooth projective model $\mathcal{X}
\to \Spec \Z$ defined in terms of the corresponding Chevalley group
schemes.  Semicontinuity of cohomology shows that vanishing of
$H^i(\mathcal{X} \times_{\Z} {\F_p}, \LL(\la_{w'} - \la_w))$ implies
the vanishing of the analogous cohomology group over $\Q$, which
implies \eqref{BWB.wts} by the characteristic zero case.
\end{proof}

\begin{rem}
The reverse implication statement of Proposition~\ref{BWB} in $\car k >0$
does not hold in general.  Indeed, take $G = \SL_3$ over a field $k$
of prime characteristic $p$.  Write $\omega_1, \omega_2$ for the
fundamental dominant weights and put $\mu = -(p+2)\omega_1 +
p\omega_2$.  For the partially ordered set of line bundles $\LL(0) <
\LL(\mu)$, \eqref{BWB.wts} obviously holds, but \eqref{BWB.bundles}
fails because both $H^1(\mu)$ and $H^2(\mu)$ are nonzero by
\cite[Cor.~5.1]{Griffith}.
\end{rem}

\begin{dfn}
A collection of weights $\{\lambda_w\}_{w\in W}$ is called
\linedef{$\PO$-exceptional} (resp.\ of the \linedef{expected length})
if the corresponding collection of line bundles $\{\LL(\lambda_w)\}_{w\in
W}$ is thus. 
\end{dfn}

The proof of Theorem~\ref{posAB} consists of two steps.  First, we
find a maximal $\PO$-exceptional subcollection of weights among the
weights $\lambda_w$ constructed in \S\ref{basweights}. This is done by
direct computations using Proposition~\ref{BWB}\eqref{BWB.wts}.  Then
we modify the remaining weights to fit in the collection, i.e., to
satisfy Proposition~\ref{BWB}\eqref{BWB.wts} and to remain a basis.
This last point is guaranteed, since we modify the weights according
to the following fact.

\begin{lem} \label{mainproc}
Let $\B$ be a basis of $\zz[\Lambda]$ over $\zz[\Lambda]^W$ and let
$e^\lambda\in \B$ be such that for some $W$-invariant set
$\{\lambda_1, \lambda_2,...,\lambda_k\}$ we have
$e^{\lambda+\lambda_i} \in \B$ for all $i<k$ and
$e^{\lambda+\lambda_k}\not\in \B$.
Then the set
\[
\bigl(\B\cup\{e^{\lambda+\lambda_k}\}\bigr)\smallsetminus\{e^{\lambda+\lambda_1}\}
\]
is also a basis of $\zz[\Lambda]$ over $\zz[\Lambda]^W$.
\end{lem}
\begin{proof}
Indeed, there is a decomposition
$$
e^{\lambda+\lambda_k}=(e^{\lambda_1}+e^{\lambda_2}+\dots+e^{\lambda_k})e^\lambda-e^{\lambda+\lambda_1}-e^{\lambda+\lambda_2}-\dots-e^{\lambda+\lambda_{k-1}}
$$
with the coefficients from $\zz[\Lambda]^W$ and an invertible coefficient at $e^{\lambda+\lambda_1}$. 
\end{proof}

We can give a geometric description of this fact.  For instance, in
type $B_2$, the rule says that if we have a square (shifted orbit of a
fundamental weight) where the center and three vertices are the basis
weights, then replacing one of these basis weights by the missing
vertex gives a basis; see Figure~\ref{B2.fig}.  For $G_2$ we
use a hexagon (shifted orbit of a fundamental weight) instead of the
square, where the center and all but one vertex are the basis weights.
\begin{figure}[hbt]
\begin{center}
\begin{picture}(350,220)
	\put(10,140){\line(1,0){140}}
	\put(80,10){\line(0,1){200}}
    \put(10,70){\line(1,1){140}}	
	\put(150,70){\line(-1,1){140}}
	{\thicklines 
	\put(80,140){\vector(0,1){60}}
	\put(80,140){\vector(1,1){30}}
	\put(115,167){$\omega_1$}		
	\put(83,203){$\omega_2$}			
	}
	\multiput(20,10)(0,15){14}{\line(0,1){7}}
	\multiput(140,10)(0,15){14}{\line(0,1){6}}
	\multiput(10,80)(15,0){10}{\line(1,0){6}}		
	\multiput(10,200)(15,0){10}{\line(1,0){6}}		
	\multiput(10,20)(15,0){10}{\line(1,0){6}}		
	\multiput(10,10)(14,14){10}{\line(1,1){10}}		
	\multiput(10,130)(14,14){5}{\line(1,1){10}}
	\multiput(70,10)(14,14){6}{\line(1,1){10}}							
	\multiput(150,10)(-14,14){10}{\line(-1,1){10}}								
	\multiput(150,130)(-14,14){6}{\line(-1,1){10}}									
	\multiput(90,10)(-14,14){6}{\line(-1,1){10}}									
	\put(80,140){\circle*{6}}
	\put(80,80){\circle*{6}}
	\put(140,140){\circle*{6}}
	\put(110,110){\circle*{6}}
	\put(50,50){\circle*{6}}
	\put(50,110){\circle*{6}}
	\put(50,170){\circle*{6}}
	\put(20,140){\circle*{6}}				

    \put(50,110){\oval(80,80)}

	\put(160,110){\vector(1,0){30}}

	\put(200,140){\line(1,0){140}}
	\put(270,10){\line(0,1){200}}
    \put(200,70){\line(1,1){140}}	
	\put(340,70){\line(-1,1){140}}
	\multiput(210,10)(0,15){14}{\line(0,1){7}}
	\multiput(330,10)(0,15){14}{\line(0,1){6}}
	\multiput(200,80)(15,0){10}{\line(1,0){6}}		
	\multiput(200,200)(15,0){10}{\line(1,0){6}}		
	\multiput(200,20)(15,0){10}{\line(1,0){6}}		
	\multiput(200,10)(14,14){10}{\line(1,1){10}}		
	\multiput(200,130)(14,14){5}{\line(1,1){10}}
	\multiput(260,10)(14,14){6}{\line(1,1){10}}							
	\multiput(340,10)(-14,14){10}{\line(-1,1){10}}								
	\multiput(340,130)(-14,14){6}{\line(-1,1){10}}									
	\multiput(280,10)(-14,14){6}{\line(-1,1){10}}	
	
	{\thicklines 
	\put(270,140){\vector(0,1){60}}
	\put(270,140){\vector(1,1){30}}
	\put(305,167){$\omega_1$}		
	\put(273,203){$\omega_2$}			
	}
								
    \put(240,110){\oval(80,80)}
	
	\put(270,140){\circle*{6}}
	\put(270,80){\circle*{6}}
	\put(330,140){\circle*{6}}
	\put(300,110){\circle*{6}}
	\put(240,50){\circle*{6}}
	\put(240,110){\circle*{6}}
	\put(240,170){\circle*{6}}
	\put(210,80){\circle*{6}}				
\end{picture}
\caption{Example of the substitution described in \S\ref{mainproc} in
the $B_2$ weight lattice.  The thick points represent the basis, and
the solid lines are the walls of Weyl chambers.}
\label{B2.fig}
\end{center}
\end{figure}
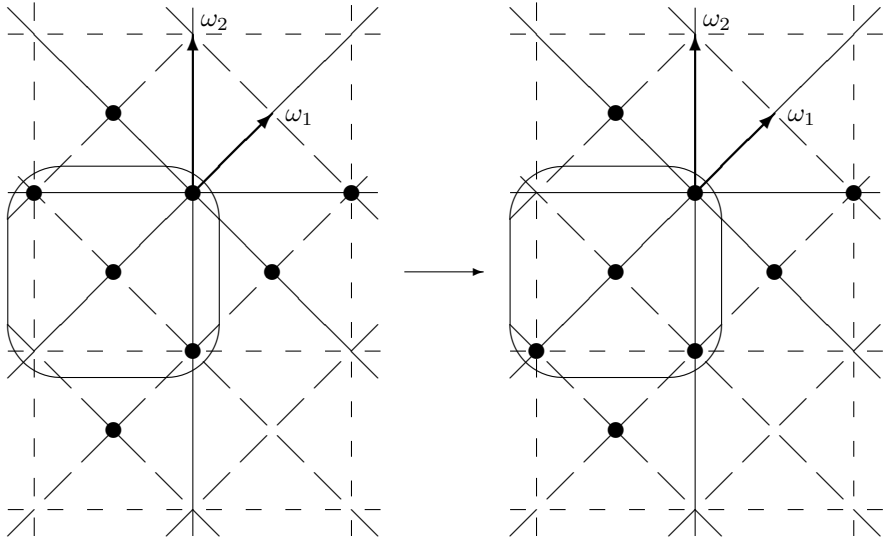

\medskip

\begin{proof}[Proof of Theorem \ref{posAB}]
We use the following notation:
a product of simple reflections 
$w=s_{i_1}s_{i_2}\cdots  s_{i_k}$ is denoted
by $[i_1,i_2,\ldots,i_k]$; the identity is denoted by $[]$.
Given a presentation $\lambda=a_1\omega_1 + \cdots +a_n\omega_n$ in 
terms of fundamental weights we denote $\lambda$ by $(a_1,\ldots,a_n)$.
We write $\PO$ for the left weak Bruhat order on the Weyl group $W$.

\medskip
\paragraph{\bf Type $A_2$:}
The Weyl group $W$ of type $A_2$ consists of the following elements:
\[
W=\{[], [1], [2], [2, 1], [1, 2], [1, 2, 1]\}.
\]
The respective basis weights $\{\lambda_w\}_{w\in W}$ of
\S\ref{basweights} are given by (here the $i$-th weight is indexed by
the $i$-th element of the Weyl group):
\begin{equation}\label{wb-a2}
\{(0,0),(-1,1),(1,-1),(-1,0),(0,-1),(-1,-1)\}.
\end{equation}
Direct computations using Proposition~\ref{BWB}\eqref{BWB.wts} show that
$\{\lambda_w\}_{w\in W}$ is a $\PO$-exceptional collection.

Using Lemma~\ref{mainproc} we can modify the weights $\{\lambda_w\}$
to obtain the following exceptional collection of weights (there are
no $\Ext$'s from left to the right):
\begin{equation}\label{tot-a2}
\{(0,0),(-1,0), (-2,0),(1,-1),(0,-1),(-1,-1)\}.
\end{equation}

\medskip
\paragraph{\bf Type $A_1 \times A_1$:}
The Weyl group of type $A_1 \times A_1$ is the Klein four-group; the root system has orthogonal simple roots $\alpha_1, \alpha_2$ which may be taken to have square-length 2 and equal the fundamental weights.  The procedure from \S\ref{basweights} gives the list of basis weights $0, -\alpha_1, -\alpha_2, -\alpha_1 - \alpha_2$, which is an exceptional collection.

\medskip
\paragraph{\bf Type $B_2$ or $C_2$:}
The Weyl group $W$ consists of the following elements:
\[
W=\{[], [1], [2], [2, 1], [1, 2], [1, 2, 1], [2, 1, 2], [1, 2, 1, 2]\}.
\]
The respective basis weights $\{\lambda_w\}_{w\in W}$ of
\S\ref{basweights} are given by:
\[
\{(0,0),(-1,1),(2,-1),(-2,1),(1,-1),(-1,0),(0,-1),(-1,-1)\}
\]
(here $\omega_1=(e_1+e_2)/2$ and $\omega_2=e_2$).  Direct computations
show that $\{\lambda_w\}_{w\in W}$ is a $\PO$-exceptional collection
except of the weight $\lambda_{[2,1]}=(-2,1)$: indeed, the property in
Proposition~\ref{BWB}\eqref{BWB.wts} fails only for the weights $0$ and
$(-2,1)$, i.e., $\lambda_{[2,1]}+\rho$ is regular or, equivalently,
$[1]*\lambda_{[2,1]}=0$ is dominant.

We modify the basis weights using Lemma~\ref{mainproc}: An element
$e^{(-2,0)}$ has the following representation with respect to the
initial basis:
$$
e^{(-2,0)}=(e^{(1,0)}+e^{(1,-1)}+e^{(-1,1)}+e^{(-1,0)})e^{(-1,0)}-e^{(0,0)}-e^{(-2,1)}-e^{(0,-1)}
$$
where $(e^{(1,0)}+e^{(1,-1)}+e^{(-1,1)}+e^{(-1,0)})\in \zz[\Lambda]^{W}$.
Hence, we can substitute $e^{(-2,1)}$ by $e^{(-2,0)}$. 
Figure \ref{B2.fig} illustrates our arguments.
Finally, after reindexing  we obtain a $\PO$-exceptional collection
\begin{equation}\label{wb-b2}
\{(0,0),(-1,1),(2,-1),(-1,0),(1,-1),(-2,0),(0,-1),(-1,-1)\}.
\end{equation}

\smallskip

Repeating Lemma~\ref{mainproc} we obtain the following exceptional
collection of weights
\begin{equation}\label{tot-b2}
\{(1,0), (0,0), (-1,0), (-2,0), (2,-1), (1,-1), (0,-1), (-1,-1)\}.
\end{equation}

\paragraph{\bf Type $G_2$:}
The Weyl group $W$ of type $G_2$ consists of the following 12 elements:
\begin{align*}
W=\{&[],[1],[2],[2,1],[1,2],[1,2,1],[2,1,2],[2,1,2,1],[1,2,1,2],\\
&[1,2,1,2,1],[2,1,2,1,2],[1,2,1,2,1,2]\}.
\end{align*}
The respective basis weights $\{\lambda_w\}_{w\in W}$ of
\S\ref{basweights} are given by:
\begin{gather*}
\{(0,0),(-1,1),(3,-1),(-3,2),(2,-1),(-2,1),(3,-2),\\
(-3,1),(1,-1),(-1,0),(0,-1),(-1,-1)\}.
\end{gather*}
Using Lemma~\ref{mainproc} we obtain the following $\PO$-exceptional
collection
\begin{gather}
\{(0,0),(-1,1),(3,-1),(-3,1),(2,-1),(-1,0),(1,-1), \label{wb-g2} \\
(-2,0),(0,-1),(-3,0),(2,-2),(-1,-1)\}. \notag
\end{gather}
Indeed, the difference between the initial basis and the modified one consists in the substitution of 
$e^{(-3,2)}, e^{(-2,1)}, e^{(3,-2)}$ by $e^{(-3,0)},e^{(-2,0)},e^{(2,-2)}$. We proceed in several steps using the same reasoning as in the $B_2$-case. Denote
\begin{align*}
A &=e^{(1,0)}+e^{(2,-1)}+e^{(1,-1)}+e^{(-1,0)}+e^{(-2,1)}+e^{(-1,1)},\\
B &=e^{(0,1)}+e^{(3,-1)}+e^{(3,-2)}+e^{(0,-1)}+e^{(-3,1)}+e^{(-3,2)},
\end{align*}
the sums of the elements corresponding to the orbits of the fundamental weights. Note that $A,B\in \zz[\Lambda]^W$.
We have the following decompositions with respect to the initial basis (coefficients belong to $\zz[\Lambda]^W$):
\begin{align*}
e^{(2,-2)}&=Ae^{(1,-1)}-e^{(0,0)}-e^{(2,-1)}-e^{(3,-2)}-e^{(0,-1)}-e^{(-1,0)}, \\
e^{(-4,2)}&=Ae^{(-2,1)}-e^{(0,0)}-e^{(-1,0)}-e^{(-3,1)}-e^{(-3,2)}-e^{(-1,1)}.
\end{align*}
Hence we can substitute $e^{(2,-2)}$ for $e^{(3,-2)}$ and $e^{(-4,2)}$ for $e^{(-3,2)}$. Using the decomposition
\[
e^{(-2,0)}=Ae^{(-1,0)}-e^{(0,0)}-e^{(1,-1)}-e^{(-1,0)}-e^{(-3,1)}-e^{(-2,1)}
\]
we substitute $e^{(-2,0)}$ for $e^{(-2,1)}$. Then, using
\[
e^{(-4,1)}=Be^{(-1,0)}-e^{(-4,2)}-e^{(-1,1)}-e^{(2,-1)}-e^{(2,-2)}-e^{(-1,-1)}
\]
substitute $e^{(-4,1)}$ for $e^{(-4,2)}$. At last, using
\[
e^{(-3,0)}=Ae^{(-2,0)}-e^{(-1,0)}-e^{(0,-1)}-e^{(-1,-1)}-e^{(-4,1)}-e^{(-3,1)}
\]
we substitute $e^{(-3,0)}$ for $e^{(-4,1)}$, obtaining the required basis.
\end{proof}

\section{Full exceptional collections of line bundles in
rank $\leq 2$} 
\label{summary.sec}

We now provide the answer to the question
\begin{equation} \label{ques}
\parbox{4in}{
Does $X$ have a full exceptional collection of line bundles?}
\end{equation}
for each projective homogeneous $G$-variety $X$, where $G$ is split
semisimple of rank $\le 2$.  In type $B_2=C_2$, we assume that the
base field has characteristic $\neq 2$.

\subsection*{Types $A_1$ and $A_2$} In case $G$ has type $A$, it is
isogenous to $\SL_2$ or $\SL_3$.  In all but one case, $X$ is
projective space $\P^1$ or $\P^2$ and the answer to \eqref{ques} is
``yes" by Be{\u\i}linson \cite{beilinson1} and
Bern{\v{s}}te{\u\i}n--Gelfand--Gelfand \cite{BGG} (which holds over
any field, cf.\ \cite[\S11]{Walter}).  In the remaining case, $G$ is
isogenous to $\SL_3$ and $X$ is the variety of Borel subgroups of $G$;
then $X$ is a $\P^1$-bundle over $\P^2$ and the answer is ``yes" by
\cite[Cor.~2.7]{Orlov} (which also holds over any field).

\subsection*{Type $A_1 \times A_1$} In this case, $G$ is isogenous to
$\SL_2 \times \SL_2$, and $X$ is either $\P^1$ or $\P^1 \times \P^1$.
In both cases, the answer to \eqref{ques} is ``yes'' since we know the
answers for projective space and products of projective spaces, as in
the previous paragraph.

\subsection*{Type $G_2$} For $G$ of type $G_2$, the answer is always
``no'' by Theorem \ref{neg}, which holds over any field.

\subsection*{Type $B_2 = C_2$} In this case, $G$ is isogenous to
$\SO_3$ and to $\Sp_4$.  Two of the possibilities for $X$ are $\P^3$
(for which the answer is ``yes" as in the type $A_2$ case) and the
Borel variety $X$ (which is a $\P^1$-bundle over $\P^3$, so that the
answer is again ``yes" as in the type $A_1$ and $A_2$ case).  The only
remaining possibility for $X$ is a 3-dimensional quadric, for which
the answer is ``no" by the following:

\begin{prop}\label{propexist}
Let $X$ be a smooth 3-dimensional quadric defined over a field of
characteristic $\neq 2$.  Then $K_0(X)$ is not generated by line
bundles.  In particular, $X$ does not have a $\PO$-exceptional
collection of the expected length consisting of line bundles for any
partially ordered set.
\end{prop}

\begin{proof}
We may assume that $G$ is simply connected; we write $\Gb$ for the adjoint group.  Put $P$ for a standard parabolic subgroup of $G$ so that $X \simeq G/P$.  We use the identification
$$
\Z \otimes_{\Z[\Lambda]^W} \Z[\Lambda]^{W_P} = K_0(X)
$$
where $\Lambda$ is the weight lattice, $W_P$ is the Weyl of the Levi
part of $P$
and $\Z[\Lambda]^{W_P}$ and  $\Z[\Lambda]^W$ are the representation rings of $P$ and $G$ respectively.

For sake of contradiction, suppose $K_0(X)$  is generated by line bundles.
Since $\Pic(X)=\Z$ is generated by $\LL(\omega)$, where $\omega_1$
is the first fundamental weight (numbered as in \cite{Bou:g4}), all such bundles are powers of $\LL(\omega_1)$, i.e.,
are of the form $\LL(n\omega_1)$ for some $n\in \Z$.

There is a surjective homomorphism 
$\pi :K_0(X) \to K_0(\Gb)$ induced by $\Lambda \to \Lambda/\Lambda_r=\Z/2\Z={\langle\sigma\rangle}$, 
where $\Lambda_r$ is the root lattice.
By \cite[Example 3.7]{Z:twisted},
$K_0(\Gb)=\Z[y]/(y^2-2y,4y)$, where $y=1-e^\sigma$.

Now take the vector bundle $E$ corresponding to the $W_P$-orbit of the
second fundamental weight $\omega_2$, i.e., to $e^{\omega_2} +
e^{\omega_2-\alpha_2}$ (here $\alpha_2$ is the second simple root),
and observe that $\pi(\LL(n\omega_1))=1$ and $\pi(E)=2-2y$.  This is a
contradiction, as $\pi(E)$ can't be written as a linear combination of
$1$'s.
\end{proof}

\part{Exceptional collection of line bundles on $G_2$-varieties} \label{negative}

In this part, we study exceptional collections of line bundles on the
Borel variety $X$ of a group $G$ of type $G_2$ over an arbitrary
field.

\begin{eg}
Suppose that $G$ is split and $\car k = 0$.  One of the projective
homogeneous $G$-varieties is a 5-dimensional quadric $Y$ and $X \to Y$
is a $\P^1$-bundle.  An exceptional collection of vector bundles on
$Y$ described in \cite{kapranov:quadric} and
\cite{kapranov:homogeneous} includes 5 line bundles.  These yield
an exceptional collection of line bundles on $X$ of length 10 by
\cite[Cor.~2.7]{Orlov}.
\end{eg}

As $K_0(X)$ is a free module of rank 12 (the order of the Weyl
group), the collection provided in the preceding example is not of
expected length.  Nonetheless, we prove that it is ``best possible":
 
\begin{thm} 
\label{G2} 
Every exceptional collection of line bundles on the Borel variety of a
group of type $G_2$ has length $\le 10$.
\end{thm}

The proof will occupy the rest of the paper.  But for now, we note
that this theorem is sufficient to prove Theorem \ref{neg}.

\begin{proof}[Proof of Theorem \ref{neg}]
By Proposition \ref{prop:scalar_extensions}, we may assume that the
group $G$ of type $G_2$ is split.  For $X$ the variety of Borel
subgroups of $G$, $K_0(X)$ is isomorphic to $\Z^{12}$ and Theorem
\ref{G2} says that there does not exist an exceptional collection of
the expected length consisting of line bundles.  Any other projective
homogeneous variety $Y$ for $G$ can be displayed as the base of a
$\P^1$-bundle $X \to Y$ and $K_0(Y)$ is a free $\Z$-module of rank 6.
Hence any exceptional collection of the expected length consisting of
line bundles on $Y$ lifts to an exceptional collection of the expected
length consisting of line bundles on $X$ by \cite[Cor.~2.7]{Orlov}
(which holds over any noetherian base scheme, cf.\
\cite[\S11]{Walter}).
\end{proof}

The group $K_0(X)$ depends neither on the base field nor on the
particular group $G$ of type $G_2$ under consideration.  Combining
this with Proposition \ref{prop:scalar_extensions}, in order to prove
Theorem \ref{G2} we may assume that the base field is algebraically
closed and hence that $G$ is split.  Proposition~\ref{BWB} reduces the
proof to a computation with totally ordered lists of weights, which we
write in ascending order.  We will use without much comment that, for
any exceptional collection $\la_1, \ldots, \la_n$ and any weight
$\mu$, the lists
\begin{equation} \label{transverse}
 \la_1 - \mu, \la_2 - \mu, \ldots, \la_n - \mu \quad \text{and} \quad {-\la_n}, {-\la_{n-1}}, \ldots, {-\la_1}
 \end{equation}
are also exceptional collections.  Thus, given an exceptional
collection, we obtain another exceptional collection of the same
length but with first entry $\la_1 = 0$.  Note also the trivial fact
that $\la_i \ne \la_j$ for all $i \ne j$, since $\Ext^*(\LL(\la_i),
\LL(\la_i)) = k$ for every weight $\la_i$ as in the proof of
Proposition~\ref{BWB}.

\section{A dichotomy}

The \linedef{crab} is the collection of weights $\la$ of $G_2$ such that $\la +
\rho$ is singular.  The crab consists of weights lying on 6 \linedef{crab
lines}.  Any pair of lines meets only at $-\rho$, and $-\rho$ lies on
all 6 crab lines.  The weight zero is not on any crab line. See
Figure~\ref{20points.fig} for a picture of the crab.
\begin{figure}[thb]
\begin{center}
\includegraphics[width=4in]{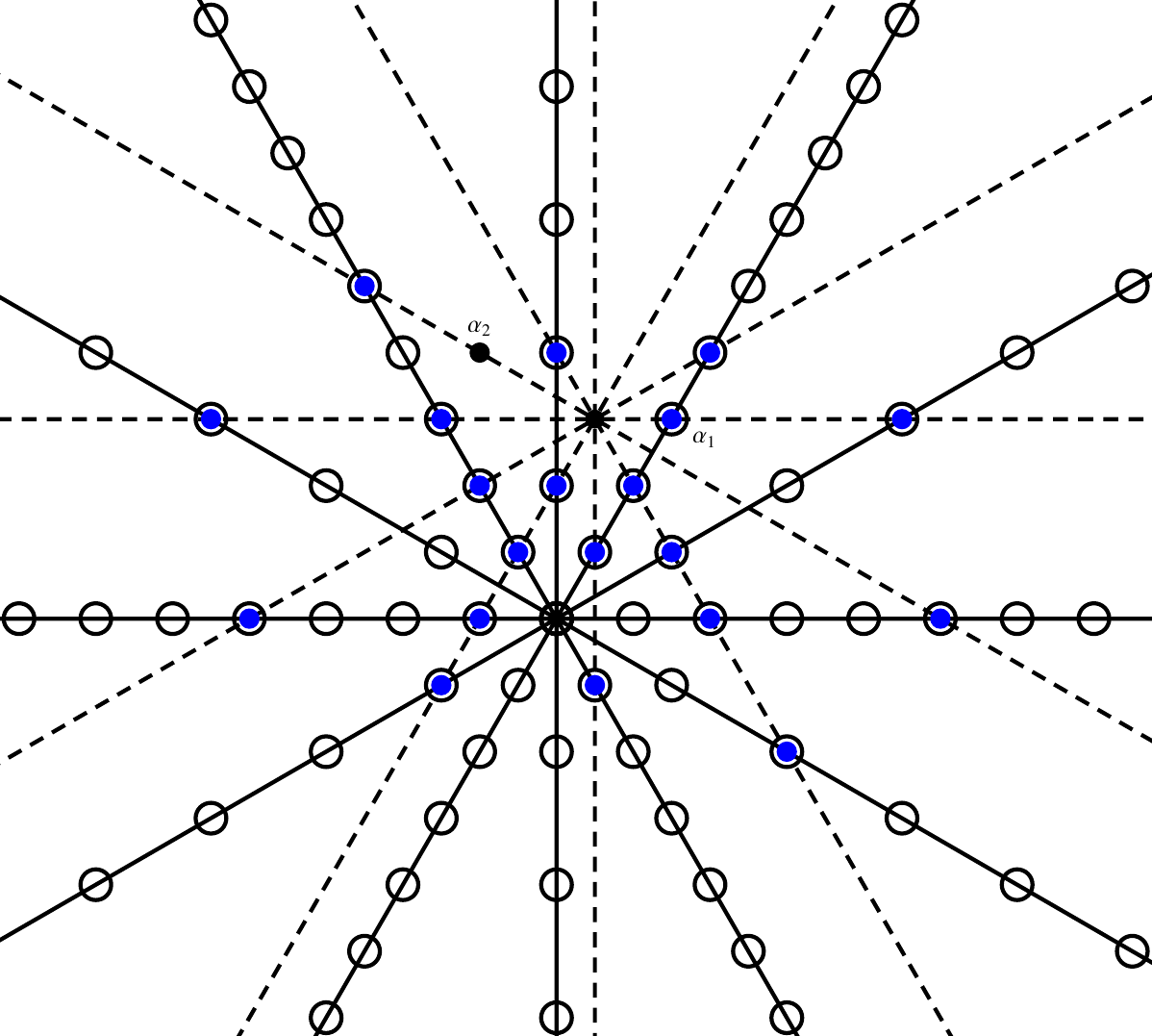} 
\caption{The crab lines (solid), the weights of the crab (circles),
the singular lines (dashed lines), and the 20 weights (disks).  All
singular lines meet at $0$ and all crab lines meet at $-\rho$.
The simple roots are $\alpha_1 = (1, 0)$ and 
 $\alpha_2 = (-3/2, \sqrt{3}/2)$.}\label{20points.fig}
\end{center}
\end{figure}

\begin{defn}
In Figure~\ref{20points.fig}, we find 20 weights on the intersection of
the crab lines and the singular lines, i.e., there are 20 weights $\la$ such
that $\la$ and $\la + \rho$ are both singular.  We call them \linedef{the 20 weights}.  
\end{defn}

We note for future reference that for each of the 20 weights $\la$, we have $|| \la || \le 3 \sqrt{3}$ and $\| \la + \rho \| \le 3 \sqrt{3}$.

\begin{lem} \label{pts20}
Suppose $0, b, c$ is an exceptional collection.
\begin{enumerate}
\item If $b$ and $c$ lie on the same crab
line, then $c - b$ is one of the 20 weights and it is on the singular line
parallel to that crab line.
\item If $c-b$ and $c$ lie on the same crab line, then $b$ is one of the 20 weights and it is on the singular line parallel to that crab line.
\end{enumerate}
\end{lem}

\begin{proof}
We prove (2) first.  The weight $b$ is in the crab because $0, b$ is
exceptional.  Further, $c-b = tc + (1-t)(-\rho)$ for some $t \in \R$,
hence $b = (1-t)(c + \rho)$, which is singular because $0, c$ is
exceptional, so $b$ is one of the 20.  If $x$ is a nonzero vector
orthogonal to $c + \rho$ and $(c-b)+\rho$, then it is also orthogonal
to their difference, $b$, which proves (2).  Then (1) is deduced from (2) via \eqref{transverse}.
\end{proof}

\begin{cor} \label{LMP}
In any exceptional collection $0, \la_2, \ldots, \la_n$, the distance between any pair of weights on the same crab line is at most $3\sqrt{3}$.
\end{cor}

\begin{proof}
Fix a crab line of interest and let  $0, \la_2, \la_3, \ldots, \la_n$
be an exceptional collection.  By restricting to a sub-list, we may assume that all of the weights $\la_2, \ldots, \la_n$ lie on that crab line.
By Lemma~\ref{pts20}(1), $\la_j - \la_2$ is one of the 20 weights for $j = 3, \ldots, n$, hence
 $||\la_j - \la_2|| \le 3\sqrt{3}$, as claimed.
\end{proof}

\begin{lem}[Trigonometry] \label{trig}
If two weights on crab lines are each $\ge R$ from $-\rho$ and are closer than $2 (2 - \sqrt{3})R$ apart, then they are on the same crab line.
\end{lem}

\begin{proof}
For sake of contradiction, suppose that the two weights are on different crab lines.  The distance between the two weights is at least the length of the shortest line segment joining the two crab lines and meeting them at least $R$ from $-\rho$.  This segment is the third side of an isosceles triangle with two sides of length $R$ and internal angle $2\theta$, hence it has length $2R\tan\theta$, where $2\theta$ is the angle between the two crab lines.  As $15^\circ \le \theta \le 75^\circ$, the minimum is achieved at $\tan 15^\circ = 2 - \sqrt{3}$.
\end{proof}

The following proposition says that a close-in weight early in the exceptional collection controls the distribution of far weights on crab lines coming later in the collection.  

\begin{prop}[Dichotomy] \label{key}
Let $0, \mu, \la_3, \ldots, \la_n$ be an exceptional collection.  Then exactly one of the following holds:
\begin{enumerate}
\item $\mu$ is one of the 20 weights, and all the $\la_j$ such that $||\la_j + \rho|| > 6\sqrt{3}$ lie on the crab line parallel to the singular line containing $\mu$.
\item $\mu$ is \emph{not} one of the 20 weights and $||\la_j + \rho|| < 3||\mu||$ for all $j = 3, \ldots, n$.
\end{enumerate} 
\end{prop}

\begin{proof}
First suppose that $\mu$ is \emph{not} one of the 20 weights.  As $0, \la_j - \mu, \la_j$ is an exceptional collection, $\la_j - \mu$ is on a crab line; by Lemma \ref{pts20}(2) it is a different crab line from $\la_j$.  By the Trigonometry Lemma, 
\[
1.9 \| \mu \| \ge \frac{\| \la_j - (\la_j - \mu) \|}{2(2-\sqrt{3})} \ge \min \{ \| \la_j + \rho \|, \| \la_j - \mu + \rho \| \}.
\]
If $\| \la_j - \mu + \rho \|$ is the minimum, then $\| \la_j + \rho \| \le \| \la_j - \mu + \rho \| + \| \mu \| \le 2.9 \| \mu \|$.  This proves (2).

Suppose that $0, \mu, \la$ is an exceptional collection such that $\mu$ is one of the 20 weights and $\la$ does not lie on the crab line parallel to the $\mu$ singular line.  Translating $0, \mu, \la$ by $-\mu$ gives the exceptional collection $0, \la - \mu$ so $\la - \mu$ also belongs to the crab, i.e., $\la$ lies in the intersection of the crab and the crab shifted by $\mu$.  We will show that this implies $|| \la + \rho || \le 6\sqrt{3}$, even ignoring questions of belonging to the weight lattice.

\begin{figure}[hbt]
\includegraphics[width=3in]{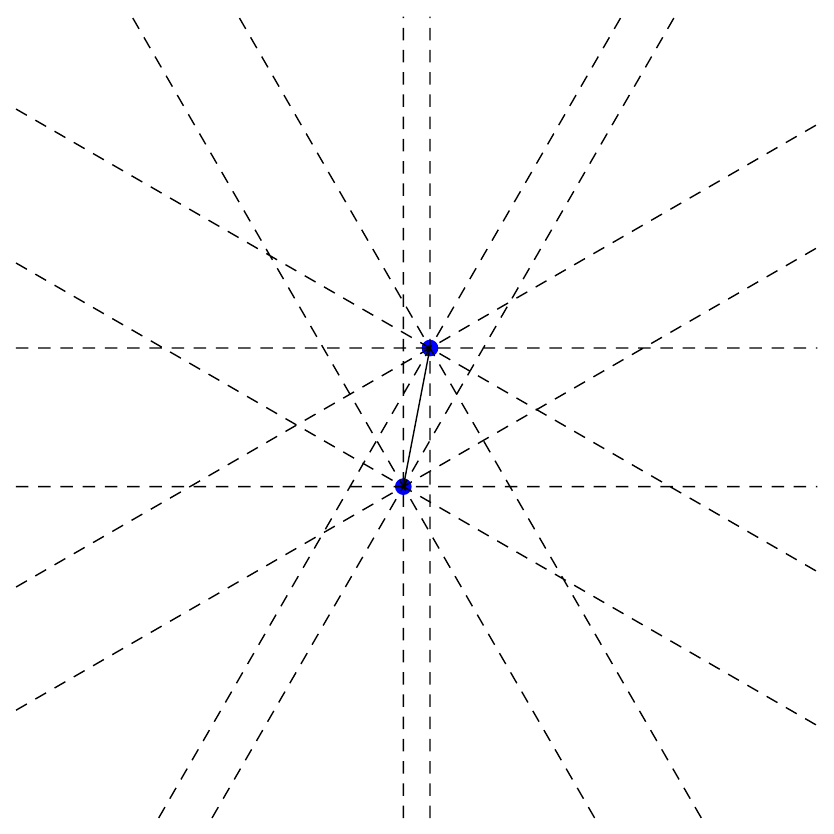}
\caption{The intersection of the crab and the crab shifted by a $\mu$ on a singular line. The solid line has length $\| \mu \|$.} \label{twocrab}
\end{figure}
Indeed, the crab and the crab shifted by $\mu$ give a picture as in Figure~\ref{twocrab}.  For each weight on the intersection of two dashed lines, we find a triangle where a side of length $||\mu||$ is bracketed by angles $\alpha, \beta$ that are multiples of 30\deg\ and $\alpha + \beta \le 150^\circ$.  Using the Law of Sines, we find that the length of the longest of the other two sides of such a triangle is
\[
\frac{\sin(\max \{ \alpha, \beta \})}{\sin(180^\circ - \alpha - \beta)} ||\mu||.
\]
Plugging in all possibilities for $\alpha$ and $\beta$, we find that the fraction has a maximum of 2.  The length of $\mu$ is at most $3\sqrt{3}$, whence the claim.
\end{proof}


\section{Three weights on two crab lines}

We now examine the possibilities for exceptional collections $0, \la_2, \la_3, \la_4$ such that two of the $\la_j$ lie on one crab line and the third lies on a different crab line; there are three possible such permutations, which we label BAA, AAB, and ABA.

\begin{defn}
A weight $a$ is \linedef{near} if $||a + \rho|| \le 42$. 
It is \linedef{far} if $\|a + \rho\| > 42 + 3 \sqrt{3}$.
\end{defn}

\begin{lem}[BAA] \label{BAA}
Suppose that $0, b, a_1, a_2$ is an exceptional collection such that
$a_1, a_2$ are on the same crab line and $b$ is neither on that crab line nor on the parallel singular line.  Then $b, a_1, a_2$ are all near weights (and in fact are within $21.1$ of $-\rho$).
\end{lem}

\begin{proof} 
Translating the exceptional collection, we find the collection $0, a_1
-b, a_2 -b$, so $a_1 - b$ and $a_2 - b$ belong to the crab.  Now,
$a_1, a_2$ are on a crab line (call it $A$) and $b$ is not on the
parallel singular line (i.e., $b$ is not parallel to $a_2 - a_1$), therefore $a_1 -b$ and $a_2 - b$ are not on the $A$ crab line.  Furthermore, because the direction $a_2 - a_1 = (a_2 - b) - (a_1 - b)$ characterizes the $A$ line, we conclude that $a_1 - b$ and $a_2 - b$ lie on different crab lines.

However,
\[
|| (a_2 - b) - (a_1 - b)|| = || a_2 - a_1 || \le 3\sqrt{3}
\]
by Corollary~\ref{LMP}, hence
\[
|| a_j - b + \rho|| \le \frac{3\sqrt{3}}{2(2-\sqrt{3})}
\]
by the Trigonometry Lemma~\ref{trig}.  By the triangle inequality
\[
|| a_j - b || \le ||a_j - b + \rho|| + ||-\rho|| \le
\frac{3\sqrt{3}}{2(2-\sqrt{3})} + \sqrt{7} < 12.4.
\]
As $a_j$ and $b$ are on different crab lines, the argument in the Trigonometry Lemma~\ref{trig} gives that $||a_j + \rho||$ and $||b + \rho||$ are at most
\[
\frac{\frac{3 \sqrt{3}}{2(2 - \sqrt{3})} + \sqrt{7}}{2(2-\sqrt{3})} <
21.04.
\qedhere
\]
\end{proof}

Here are two corollaries from the Trigonometry Lemma~\ref{trig}.
\begin{cor} \label{plusrho}
If $\la$ and $\la + \rho$ are on crab lines and $|| \la + \rho || > 7.7$, then $\la$ and $\la + \rho$ are on the same crab line.
\end{cor}

\begin{proof}
We use the triangle inequality to bound the distance of $\la + \rho$ from $-\rho$:
\[
|| \la + 2\rho || \ge || \la + \rho || - || {-\rho}|| > 7.7 - \sqrt{7} > 5.05.
\]
The distance between $\la$ and $\la + \rho$ is $||\rho|| = \sqrt{7} 
< 2(2-\sqrt{3})5.05$, so taking $R = 5.05$ in the Trigonometry Lemma~\ref{trig} gives the claim.
\end{proof}

\begin{cor} \label{crab.diff}
If $a$ and $b - a$ both lie on some crab line $A$, then so does $b + \rho$.  If furthermore $|| b + \rho || > 7.7$, then $b$ also lies on $A$.
\end{cor}

\begin{proof}
Let $x$ be a nonzero vector orthogonal to $a + \rho$ and $b - a + \rho$.  Then $x$ is also orthogonal to $(b - a + \rho) + a + \rho = (b + \rho) + \rho$; this proves the first claim.
For the second claim, we apply Corollary~\ref{plusrho}.
\end{proof}

\begin{lem}[AAB]
\label{AAB}
If $0, a_1, a_2, b$ is an exceptional collection where $a_1, a_2$ are
on one crab line and $b$ is on a different crab line, then at least
one of $a_1, a_2, b$ is near.
\end{lem}

\begin{proof}
For sake of contradiction, suppose all three nonzero weights are at least 42 from $-\rho$.  Translating, we find an exceptional collection
$0, a_2 - a_1, b - a_1$, where for $j  = 1,2$ we have
$|| b - a_j || > 2(2 - \sqrt{3}) 42.$ 
Further,
\[
|| b - a_j + \rho || \ge || b - a_j || - || {-\rho} || > 2(2 - \sqrt{3})42 - \sqrt{7} > 19.8.
\]
Now 
\[
|| (b - a_2) - (b - a_1) || = || a_1 - a_2|| \le 3\sqrt{3} 
< 2(2-\sqrt{3})19.8,
\]
so by the Trigonometry Lemma~\ref{trig} $b - a_2$ and $b - a_1$ lie on the same crab line.

As $a_1, a_2$ also lie on one crab line, we can find nonzero vectors $x, y$ such that $x$ is orthogonal to $b - a_j + \rho$ and $y$ is orthogonal to $a_j + \rho$ for $j = 2, 3$.  It follows that
\[
a_1 - a_2 = (a_1 + \rho) - (a_2 + \rho) = (b - a_2 + \rho) - (b - a_1 + \rho)
\]
is orthogonal to both $x$ and $y$.  As $a_1 - a_2 \ne 0$, it follows that the four weights $a_j, b - a_j$ for $j = 2, 3$ all lie on one crab line.   Corollary~\ref{crab.diff} gives that $b + \rho$ lies on this same line.  As $b$ is also on a crab line and $b$ is at least 42 from $-\rho$, Corollary~\ref{plusrho} gives that $b$ and $b + \rho$ are on the same crab line.  This contradicts the hypothesis that $a_1, b$ are on different crab lines.
\end{proof}

We now prepare for ABA, the most complicated of the three configurations.


\begin{defn}
The \linedef{mirror 20 weights} consist of the intersection of the crab
with the crab shifted to $-\rho$.  A weight $\mu$ is one of the mirror
20 weights if both $\mu$ and $\mu+\rho$ are in the crab.  The lines
through $-2\rho$ parallel to the crab lines will be called the
\linedef{mirror singular lines}.
\end{defn}

Now we need a ``mirror'' version of Proposition \ref{key}(2).

\begin{prop}
\label{mirror_farclose}
If $0,\la,\mu$ is an exceptional collection with $\|\la\| \geq
2.9\|\mu+\rho\| + 7.6$, then $\mu$ is one of the mirror 20 weights on
the mirror singular line parallel to $\la$. In particular, in this
case $\|\mu+\rho\| \leq 3\sqrt{3}$.
\end{prop}
\begin{proof}
As $\la$ is in the crab, $\la + \rho$ is singular, thus $-2\rho - \la
+ \rho = -\rho - \la = -(\la+\rho)$ is singular, hence $-2\rho-\la$ is
in the crab.  Also $\mu-\lambda$ is in the crab by exceptionality.
We will show that $\mu-\la$ and $-2\rho -\la$ are on the same crab
line, hence that $(\mu-\la)-(-2\rho-\la+\rho) = \mu+\rho$ is also on
the crab and thus $\mu$ is one of the mirror 20 weights.  

We have
$$
\| (\mu-\la) - (-2\rho - \la)\| = \|\mu+2\rho\| \leq \|\mu+\rho\| +\sqrt{7}
$$
by the triangle inequality.  By the Trigonometry Lemma~\ref{trig}, $\mu$ will then be
one of the mirror 20 weights as long as $\mu-\la$ and $-2\rho-\la$ are a
distance 
$$
\frac{\|\mu + \rho\| + \sqrt{7}}{2(2-\sqrt{3})}<
1.87\|\mu+\rho\|+4.94
$$ from $-\rho$.  But indeed, by hypothesis, we
have
$$
\|\mu-\la+\rho\| \geq \|\la\| - \|\mu+\rho\| \geq 1.9\|\mu+\rho\| + 7.6
$$
and
$$
\|{-2\rho}-\la+\rho\| \geq \|\la\| - \|\rho\| \geq 2.9\|\mu+\rho\| +
7.6 - \sqrt{7} > 2.9\|\mu+\rho\| + 4.96.
$$
The final claims are apparent.
\end{proof}

\begin{lem}[ABA]
\label{ABA}
Suppose that $0,a_1,b,a_2$ is an exceptional collection of far weights
such that $a_1,a_2$ are on the same crab line and $b$ is on a
different crab line.  Then the collection $0,a_1,b,a_2$ is maximal.
\end{lem}

By \linedef{maximal} we mean that there is no exceptional collection $0, \la_2, \la_3, \la_4, \la_5$ that contains $0, a_1, b, a_2$ as a sub-collection.

\begin{proof}
We have a number of cases, each of which we will deal with by
contradiction.  

\case 1 
Consider extending the collection as $0,\mu,a_1,b,a_2$.  If
$\mu$ is on the same crab line as $a_1,a_2$, then the AAB Lemma~\ref{AAB}
applied to the exceptional collection $0,\mu,a_1,b$, together with the
fact that $a_1,b$ are far, implies that $\mu$ is near.  But this is
impossible since $\mu$ and $a_1$ can be at most $3\sqrt{3}$ apart.  If
$\mu$ is on a different crab line than $a_1,a_2$, then $0,\mu,a_1,a_2$
is exceptional, which contradicts the BAA Lemma~\ref{BAA} (since
$a_1,a_2$ are far) unless $\mu$ is on the singular line parallel to
the crab line containing $a_1,a_2$.  But $0,\mu,b$ is exceptional,
$\mu$ is one of the 20 weights, and $b$ is far, hence $b$ is on the crab line
parallel to the singular line containing $\mu$.  This is impossible
since $b$ and $a_1,a_2$ are on different crab lines.

\case 2 
Similarly, consider extending the collection as
$0,a_1,b,a_2,\mu$.  If $\mu$ is on the same crab line as $a_1,a_2$,
then the exceptionality of $0,b,a_2,\mu$ contradicts the BAA
Lemma~\ref{BAA} since $b,a_2$ are far ($b$ cannot be on the singular
line parallel to the crab line containing $a_2,\mu$ because it is
far).  If $\mu$ is on a different crab line than $a_1,a_2$, then by
the AAB Lemma~\ref{AAB} and the fact that $a_1,a_2$ are far, we have
that $\mu$ is near.  By Proposition~\ref{mirror_farclose} applied to
$0,a_1,\mu$ and $0,b,\mu$, we have that $\mu$ is one of the mirror 20
weights on the mirror singular line parallel to the crab lines
containing $a_1$ and $b$, hence $\mu = -2\rho$, contradicting the
hypothesis that $\mu$ is in the crab.

\case 3 
Now consider extending the collection as $0,a_1,\mu,b,a_2$.
If $\mu$ is on the same crab line as $a_1,a_2$, then the AAB
Lemma~\ref{AAB} applied to the exceptional collection $0,a_1,\mu,b$,
together with the fact that $a_1,b$ are far,  implies that $\mu$ is
near.  But this is impossible since $\mu$ and $a_1$ can be at most
$3\sqrt{3}$ apart.
Similarly, if $\mu$ is on the same crab line as $b$, then
the AAB Lemma~\ref{AAB} applied to the exceptional collection $0,\mu,b,a_2$, together with the
fact that $b,a_2$ are far, implies that $\mu$ is near.  But this is
impossible since $\mu$ and $b$ can be at most $3\sqrt{3}$ apart.  

Thus
$\mu$ is not on the crab line containing any of $a_1,b,a_2$.  
As in \eqref{transverse}, the collection
\[
0,\, a_2-b,\, a_2-\mu,\, a_2-a_1,\, a_2
\]
 is exceptional.  By Lemma
\ref{pts20}, $a_2-a_1$ is one of the 20 weights, in particular
$\|a_2-a_1+\rho\| \leq 3\sqrt{3}$.  Since $b,a_2$ are far and on
different crab lines, $\|b+\rho\|,\|a_2+\rho\| \geq 42 + 3\sqrt{3}$,
hence by the Trigonometry Lemma~\ref{trig}, we have
$$
\|a_2-b\| \geq 2(2-\sqrt{3})(42+3\sqrt{3}) 
> 2.9\cdot  3\sqrt{3} + 7.6. 
$$ 
Applying Proposition~\ref{mirror_farclose} to the exceptional
collection $0,a_2-b,a_2-a_1$, we have that $a_2-a_1$ is one of the
mirror 20 weights on the mirror singular line parallel to the crab line
containing $a_2-b$.  In particular, $\|a_2-a_1+\rho\| \leq \sqrt{3}$.

The first option of the Dichotomy Proposition~\ref{key} applied to the
exceptional collection $0,\mu,b,a_2$ is impossible since $b,a_2$ are
far and on different crab lines.  Hence by Dichotomy, $\|\mu\| >
\frac{1}{3}(42 + 3\sqrt{3}) > 15$.  In particular $\|\mu+\rho\| >
15-\sqrt{7}$, so by the Trigonometry Lemma~\ref{trig} (using that
$a_2$ is far), we have $\|a_2-\mu\| > 2(2-\sqrt{3})(15-\sqrt{7}) > 7 >
3\sqrt{3}$, in particular $a_2-\mu$ is not one of the 20 weights.

It follows that the second option in Dichotomy holds for $0,a_2-\mu,a_2$ 
and $42 + 3 \sqrt{3} < \|a_2+\rho\| < 3 \|a_2-\mu\|$.  But then we
have
$$
\|a_2-\mu\| > \frac{1}{3}(42+3\sqrt{3}) > 15 > 2.9\cdot \sqrt{3} + 7.6
$$ 
so that we can apply Proposition~\ref{mirror_farclose} to the exceptional
collection $0,a_2-\mu,a_2-a_1$.  We conclude that $a_2-a_1$ is one of
the mirror 20 weights on the mirror singular line parallel to the crab
line containing $a_2-\mu$.  

Since $b$ is far and $\mu$ is on a different crab line, the
Trigonometry Lemma~\ref{trig} says that 
\[
\| \mu - b \| > \frac{15 - \sqrt{7}}{2(2-\sqrt{3})} > 3 \sqrt{3},
\]
so $a_2 - \mu$ and $a_2 - b$ cannot lie on the same crab line.  
But then it is impossible for $a_2-a_1$ to be on the mirror singular
lines parallel to the crab lines of both $a_2-\mu$ and $a_2-b$.
Therefore no such $\mu$ can exist.

\case 4
Finally, consider extending the collection as $0,a_1,b,\mu,a_2$.  If
$\mu$ is far, then by interchanging the roles of $\mu$ and $b$, we can
use the previous argument.  Hence we can assume $\mu$ is not far. If
$\mu$ is on the same crab line as $a_1,a_2$, then the exceptionality
of $0,b,\mu,a_2$ contradicts the BAA Lemma~\ref{BAA} since $b,a_2$ are
far (in particular, $b$ cannot be on the singular line parallel to the
crab line containing $a_2,\mu$).
Similarly, if $\mu$ is on the same crab line as $b$, then the
exceptionality of $0,a_1,b,\mu$ contradicts BAA Lemma~\ref{BAA}.  Thus
$\mu$ is not on the crab line containing any of $a_1,b,a_2$.  If $\mu$
is one of the 20 weights, then we can apply Proposition
\ref{mirror_farclose} to the exceptional collections $0,a_1,\mu$ and
$0,b,\mu$ (since $42+3\sqrt{3} \geq 2.9 \cdot 3\sqrt{3}+7.6$),
concluding that $\mu$ is also one of the mirror 20 weights contained on
the mirror singular lines parallel to crab lines containing $a_1,b$,
which is impossible.  Hence as before, the Dichotomy Proposition
\ref{key} implies that $\|a_2-\mu\| > 7$.  As in \eqref{transverse}, the collection
$0,a_2-\mu,a_2-b,a_2-a_1,a_2$ is exceptional and we can use the
previous argument.  

We have thus ruled out all possible exceptional
extensions of $0,a_1,b,a_2$.
\end{proof}

Combining the Lemmas \ref{BAA}, \ref{AAB}, and \ref{ABA} gives the following:

\begin{prop}
\label{triplet} 
Suppose that $a_1, a_2, b$ are far weights such that $a_1, a_2$ lie on the
same crab line and $b$ lies on a different crab line.  Then:
\begin{enumerate}
\item Neither $0, b, a_1, a_2$ nor $0, a_1, a_2, b$ are exceptional collections.
\item If $0, a_1, b, a_2$ is an exceptional collection, then it is maximal.$\hfill\qed$
\end{enumerate}
\end{prop}

\section{Computer calculations}

Our proof of Theorem \ref{G2} makes use of the following concrete facts, which can be easily verified by computer:
\begin{fact} \label{noDMZ}
Every exceptional collection $0, \la_2, \ldots, \la_n$ with all $\la_j$ non-far has  $n \le 10$.
\end{fact}

\begin{fact} \label{forty}
Let $A$ be a crab line and $S$ the set of weights on the
union of the singular line and the mirror singular line parallel to $A$.  Then
every exceptional collection $0, \la_2, \ldots, \la_n$ of weights in
$S$ has $n \ge 5$.  (Note that as the $\la_j$'s belong to the crab, they are all selected from the union of the 20 weights and the mirror 20 weights.)
\end{fact}

\begin{fact} \label{close}
If $0, \la_2, \ldots, \la_n$ is an exceptional collection with $n = 9$ or 10, with all weights non-far, with all crab lines containing at most 2 weights, and with one crab line containing no weights, then $\| \la_j + \rho \| \le 5$ for all $j = 2, \ldots, n$.
\end{fact}

We used Mathematica 8.0.4 to check these facts.  We first wrote a function \texttt{IsSingular} that returns \texttt{True} if a weight is singular and \texttt{False} otherwise.  With the following code, and lists of weights \texttt{L1} and \texttt{L2}, the command \texttt{FindCollections[L1, L2]} will fill the global variable \texttt{collections} with a list of all of the maximal exceptional collections that begin with \texttt{L1} and such that all weights following \texttt{L1} come from \texttt{L2}.  In the code, \texttt{rho} denotes the highest root written in terms of the fundamental weights.  We omit the sanity checks that ensure that for the initial values of \texttt{L1} and \texttt{L2}, appending each element of \texttt{L2} to \texttt{L1} results in an exceptional collection.
\begin{verbatim}
collections = {}; 
FindCollections[L1_, L2_] := Module[{tmpL1},
   If[Length[L2] == 0, AppendTo[collections, L1],
    Do[
     tmpL1 = Append[L1, L2[[i]]];
     FindCollections[tmpL1, 
      Select[Delete[L2, i], IsSingular[# - L2[[i]] + rho] &]], 
      {i, 1, Length[L2]}]]];
\end{verbatim}

For example, to check Fact~\ref{noDMZ}, we constructed the list \texttt{L2} consisting of all 445 non-far weights in the crab and executed \texttt{FindCollections[\{\{0, 0\}\}, L2]} to obtain the list of the $160,\!017$ maximal exceptional collections $0, \la_2, \ldots, \la_n$ with all $\la_j$ non-far.  With this list in hand, it is not difficult to select out collections meeting the criteria of Facts~\ref{forty} and \ref{close}.
\section{Bounding exceptional collections}

This section will complete 
the proof of Theorem~\ref{G2}.

\begin{lem} \label{maxpts}
In any exceptional collection $0, \la_2, \ldots, \la_n$, at most $5$ of the $\la_j$'s lie on any given crab line.
\end{lem}

\begin{proof}
Fix a crab line of interest and let  $0, \la_2, \la_3, \ldots, \la_n$
be an exceptional collection.  By restricting to a sub-list, we may assume that all of the weights $\la_2, \ldots, \la_n$ lie on that crab line.
By Corollary~\ref{pts20}, $\la_j - \la_2$ is one of the 20 weights for $j = 3, \ldots, n$,
and Figure \ref{20points.fig} shows that $\la_3, \ldots, \la_n$ has length at most 5 corresponding to having 6 weights on the line of interest, and that the proof is complete except in the case where the crab line makes a 120\deg\ angle with the horizontal.

For that line, we must argue that $\la_3 - \la_2, \ldots, \la_n - \la_2$ cannot be the 5 weights of the 20 depicted in the figure.  Indeed, if they were, one could translate by $\la_3 - \la_2$ to transform this to an exceptional collections $0, \la_4 - \la_3, \ldots, \la_7 - \la_3$ and $\la_j - \la_3$ must belong to the crab for $j \ge 4$.  But we can see from the figure that this does not happen for any of the 5 choices for $\la_3 - \la_2$, hence the claim.
\end{proof}

\begin{prop} \label{threep}
If $0, \la_2, \ldots, \la_n$ is an exceptional collection containing at least $3$ weights on some crab line $A$, then $n \le 10$ and all the weights off $A$ are near.
\end{prop}

\begin{proof}
In the exceptional collection $0, \la_2, \ldots, \la_n$, let
$a_1,a_2,a_3$ be three weights on $A$.  Then every nonzero weight
$b$ in the collection and off $A$ either precedes $a_2,a_3$ or follows $a_1,a_2$.

First suppose $b$ precedes $a_2,a_3$.  Then by the BAA Lemma~\ref{BAA},
either $b,a_2,a_3$ are all near or $b$ is on the singular line
parallel to $A$.  In the latter case, $b$ is one of the 20 weights and so is near.

Suppose that $b$ follows $a_1, a_2$, then shifting by $-a_1$ gives an
exceptional collection $0, a_2 - a_1, b - a_1$ where $a_2 - a_1$ is
one of the 20 weights.  By the Dichotomy Proposition~\ref{key}, we
have two possibilities:

\case{1} 
We could have that $\| b - a_1 + \rho \| \le 6\sqrt{3}$, but in that case we find that
\[
\| b - a_1 \| - \| \rho \| \le \| b-a_1 + \rho \| \le 6\sqrt{3},
\]
hence $\| b - a_1 \| \le 6\sqrt{3} + \sqrt{7}$.  But $b$ and $a_1$ lie
on different crab lines, so by the Trigonometry Lemma~\ref{trig}, we
find that $\min \{ \| b + \rho \|, \| a_1 + \rho \| \}$ is at most
$(6\sqrt{3} + \sqrt{7})/(2(2-\sqrt{3})) < 25$.  If  $\| a + \rho \| < 25$, then
\[
\| b + \rho \| \le \| a_1 + \rho \| + \| b - a_1 \| < 25 + 6\sqrt{3} + \sqrt{7} < 42
\]
and $b$ is near.  (Note that if $a_1$ is non-near, then we would have $\| b + \rho \| < 25$ and this case is impossible.)

\case{2} 
Alternately, $b-a_1$ could lie on $A$.  In that case, as $b$ is not on
$A$, Corollary~\ref{crab.diff} gives that $\| b + \rho \| \le 7.7$.  
Thus all weights in the exceptional collection off $A$ are near.

\smallskip

It remains to argue that the collection has length $\le 10$.  If the
weights on $A$ are non-far, then all the weights in the collection are
non-far and we are done by Fact~\ref{noDMZ}.  Therefore, we may assume
that some of the weights on $A$ are far, hence all weights on $A$ are
non-near.  Let $b$ be a nonzero weight in the collection that is off
$A$.  If it precedes $a_2,a_3$, then $b$ is one of the 20 weights
(because $b,a_2,a_3$ cannot all be near).  Otherwise, $b$ comes after
$a_1,a_2$ and we are in Case 2 above, so $2.9 \| b + \rho \| + 7.6 \le 29.93$; by Proposition~\ref{mirror_farclose}, $b$ is one of the mirror 20 weights and lies on the mirror singular line parallel to $A$.  Fact~\ref{forty} and Lemma \ref{maxpts} show that one cannot obtain an exceptional collection of length $> 10$.
\end{proof}


We can now conclude the proof of the main theorem.

\begin{proof}[Proof of Theorem~\ref{G2}]
For sake of contradiction, we suppose we are given an exceptional
collection $0, \la_2, \ldots, \la_{11}$.  By
Proposition~\ref{threep}, no crab line contains more
than 2 weights.

We claim that the number $F$ of far weights in the collection is 1 or 2.  Indeed, by Fact~\ref{noDMZ}, $F$ is positive.  Suppose that it is at least 3.  Then by Proposition~\ref{triplet}, all far weights lie on different crab lines, leaving $6 - F$ crab lines for the remaining $10-F$ nonzero weights; but the remaining crab lines can only hold $12-2F$ weights, which contradicts our hypothesis that $F \ge 3$; hence $F = 1$ or 2.  

We will now pick a crab line $A$ and a subset $S$ of $\la_2, \ldots, \la_{11}$ containing all the far weights and all the weights on $A$, and such that $|S| = 1$ or 2.

\case{$F = 1$}  If $F = 1$, we take $A$ to be the crab line containing the far weight and let $S$ be the set of $\la_j$'s lying on $A$; by hypothesis $|S| \le 2$.

\case{$F = 2$} If both of the far weights are on one crab line, then we take it to be $A$ and $S$ to be the set of far weights.

Otherwise, the two far weights are on different crab liens.  We claim that one of these crab lines, call it $A$, contains exactly one weight from the exceptional collection.  Indeed, otherwise there would be two crab lines each containing two weights; as all of these are non-near by Corollary \ref{LMP}, this contradicts Proposition \ref{triplet}, verifying the claim.  We take $S$ to be the far weights in the exceptional collection.

\smallskip We have found $S$ as desired, and deleting it from
the exceptional collection leaves one as in Fact \ref{close} and we
conclude that $\| \la_j + \rho \| \le 5$ for all $\la_j$ not in $S$.
If such a $\la_j$ precedes one of the far weights, then it is one of
the 20 weights by the Dichotomy Proposition~\ref{key}; if it follows
one of the far weights then it is one of the mirror 20 weights by
Proposition~\ref{mirror_farclose}.  But deleting $S$ from our
exceptional collection leaves an exceptional collection starting with
0 and containing at least $8$ nonzero, non-far weights all lying off
$A$, which contradicts Fact \ref{forty}.
\end{proof}

\providecommand{\bysame}{\leavevmode\hbox to3em{\hrulefill}\thinspace}

\end{document}